\documentclass{amsart}

\newtheorem{theorem}{Theorem}[section]
\newtheorem{lemma}[theorem]{Lemma}

\theoremstyle{definition}

\theoremstyle{remark}
\newtheorem{remark}[theorem]{Remark}

\numberwithin{equation}{section}

\begin{document}

\title{open  manifolds with asymptotically nonnegative Ricci curvature and large volume growth}

\author{Yuntao Zhang}
\address{ School of Mathematics and Statistics, Jiangsu Normal University, Xuzhou, 221116,
 P.R.China}
\email{yuntaozhang@jsnu.edu.cn}
\thanks{Supported  by   PAPD of Jiangsu Higher Education Institutions.}

\subjclass[2010]{Primary 53C20; Secondary 53C21}



\keywords{Ricci curvature, finite topological type, volume rowth}

\begin{abstract}
 In this paper, we study the topology of complete noncompact Riemannian manifolds with asymptotically nonnegative Ricci curvature and large volume growth. We prove that they have  finite topological types  under some    curvature decay and volume growth conditions. We also generize it to the manifolds with $k$-th asymptotically nonnegative Ricci curvature
by using  extensions of Abresch-Gromoll's excess function estimate.\end{abstract}
\maketitle

\section{Introduction}
A complete noncompact Riemannian manifold is said to have an asymptotically nonnegative Ricci curvature
if there exists a base point $p$,  and a   positive nonincreasing function $\lambda$ such that $\int_{0}^{+\infty}s\lambda(s)ds  < +\infty$, and the Ricci curvature of $M$ at any point $x$ satisfies
$$ Ric(x)\ge -(n-1)\lambda(d_p(x)),$$
 where $d_{p}$ is the distance to $p$.
Abresch and Gromoll \cite{1} was the first one to study this class, and they proved that such manifolds have finite topological
type if the sectional curvatures are uniformly bounded and the diameter  growth have order $o(s^{1/n})$ with respect to the base point
$p$. Recall that a manifold is said to have finite topological type if there exists a compact domain $\Omega$ with boundary such that $M\setminus\Omega$ is homeomorphic to $\partial\Omega\times[0,\infty ]$. In order to complete their theorems, Abresch and Gromoll established important  excess function estimates, which are also used by Hu, Xu \cite{5} and by Mahaman \cite{7} to prove some topological rigidity results for manifolds with asymptotically nonnegative Ricci curvature.  They are also  used as important tools for many geometers to study manifolds with  nonnegative Ricci curvature, see \cite{9}, \cite{10}, \cite{11}, \cite{13},\cite{14}, etc.

Let $B(x,r)$ denote the geodesic ball of radius $r$ and center $x$ in $M$ and $B(\overline{x},r)$ denote the similar metric ball in the simply connected noncompact complete manifold with sectional curvature $-\lambda (d_{\overline{p}}(\overline{x}))$ at the point $\overline{x}$, where $d_{\overline{p}}(\overline{x})= d(\overline{p},\overline{x})$ is the distance from $\overline{p}$ to $\overline{x}$.
From the  volume comparison theorem, which was proved by Zhu \cite{15} for the base point and by Mahaman \cite{6} for any point, we know  that the function $r\mapsto \frac{volB(x,r)}{volB(\overline{x},r)}$ is monotone decreasing. Define $$\alpha_{x}\dot= \lim_{r\rightarrow +\infty}\frac{volB(x,r)}{volB(\overline{x},r)}\; \textrm{ and } \alpha_{M}\dot=\inf_{x\in M}\alpha_{x}.$$ We say $M$ is large volume growth if $\alpha_M > 0$.

For any $r>0$, let 
$$k_{x}(r)= \inf_{M\setminus B(x,r)}K,$$
 where $K$ is the sectional curvature of $M$, and the infimum is taken over all the sections at all points on $M\setminus B(x,r)$. It is easy to see that $k_{x}(r)$ is a monotone function of $r$.

For a complete open Riemannian manifold  with  nonnegative Ricci curvature and large volume growth  $\alpha_M > 0$, assume that $k_x(r)\ge -\frac{C}{(1+r)^\alpha}$ for some $x\in M$ and all $r$, where 
$C>0$ and $\alpha\in[0,2]$ are constants. Xia \cite{13} proved that $M$ has finite topological type if there is a constant
$\epsilon=\epsilon(n, C, \alpha)>0$, such that
$$\limsup_{r\to +\infty }\left\{ \left(\frac{vol B(x, r)}{\omega_{n}r^n}-\alpha_{M}\right)r^{(n-2+\frac{1}{n})(1-\frac{\alpha}{2})}
\right\}
\leq {\epsilon}\alpha_{M}.$$

The main purpose of this note is to generize the above result to the manifolds with  asymptotically nonnegative Ricci curvature. We have the following 
\begin{theorem}
Let $M$ be an $n$-dimensional $(n\ge 3)$ complete noncompact Riemannian manifold with
\begin{equation}
Ric(x)\ge -(n-1)\lambda(d_{p}(x))   \textrm{ and }  K(x)\ge -\frac{C}{d_{p}(x)^\alpha},
\end{equation}
where $C(\lambda)=\int_{0}^{+\infty}s\lambda(s)ds < +\infty$ and $C>0$, $0\le \alpha\le 2$.
If $\alpha_p >0$, then there exists a constant 
$\epsilon=\epsilon(n, \lambda, C, \alpha)>0$, such that  $M$ has finite topological type, provided that
\begin{equation}
\limsup_{r\to +\infty }\left\{ \left(\frac{vol B(x, r)}{vol B(\bar{x},r)}-\alpha_p\right)r^{(n-2+\frac{1}{n})(1-\frac{\alpha}{2})}
\right\}
\leq {\epsilon}\alpha_p.
\end{equation}
\end{theorem}

On the other hand, Shen-Wei \cite{10} studied manifolds with nonnegative $k$th Ricci curvature outside a geodesic
ball $B(p,D)$and weak bounded geometry, i.e. $\mathcal{K}=\inf K >-\infty$, $v=\inf vol B(x,1)>0$. They proved that  there is a constant 
$c=c(n, k,\mathcal{K}, v, D)>0$ such that  $M$ has finite topological type, if the volume growth at a point $x\in M$ satisfies
$$ \limsup_{r\to +\infty }\frac{vol B(p,r)}{r^{1+1/(k+1)}}<c. $$
Here we say the $k$th Ricci curvature of $M$, for some $1\le k\le n-1$, satisfies $Ric_{(k)}(x)\ge H$,
at a point $x\in M$ if for all ($k+1$)-dimensional subspaces $V\subset T_xM$,
$$ \sum^{k+1}_{i=1}<R(e_i,v)v,e_i>\  \ge H  \textrm{ for all } v\in V,$$
where $\{e_1, \cdots,e_{k+1}\}$ is any orthonormal basis for $V$. Stimulated 
by their methods, we can extend Theorem 1.1 to the case of $k$th  asymptotically nonnegative Ricci curvature.

\begin{theorem}
Let $M$ be an $n$-dimensional $(n\ge 3)$ complete noncompact Riemannian manifold with
\begin{equation}
Ric_{(k)}(x)\ge -k\lambda(d_{p}(x)), \textrm{ for } 2\le k\le n-1,
\end{equation}
    and  $$ K(x)\ge -\frac{C}{d_{p}(x)^\alpha},$$
where $C(\lambda)=\int_{0}^{+\infty}s\lambda(s)ds < +\infty$ and $C>0$, $0\le \alpha\le 2$.
If $\alpha_p >0$, then there exists a constant 
$\epsilon=\epsilon(n, k,\lambda, C, \alpha)>0$, such that  $M$ has finite topological type, provided that
\begin{equation}
\limsup_{r\to +\infty }\left\{ \left(\frac{vol B(x, r)}{vol B(\bar{x},r)}-\alpha_p\right)r^{\frac{(n-1)k}{k+1}(1-\frac{\alpha}{2})}
\right\}
\leq {\epsilon}\alpha_p.
\end{equation}
\end{theorem}

\begin{remark}
For $k=n-1$, Theorem 1.2 is just the same one as Theorem 1.1. While for $k=1$, $M$ has asymptotically nonnegative sectional curvature. It was  proved by Abresch \cite{1} that $M$ always has finite topolological type without any additional conditions.
\end{remark}

Denote by $crit_p$ the criticality radius of $M$ at $p$, i.e. $crit_p$ is the smallest critical
value for the distance function $d_p(\cdot)$. Recall a point $x\ne p$ is called critical point of $d_{p}$ if for any $v$ in the tangent space $T_{x}M$ there is minimal geodesic $\gamma$ from $x$ to $p$ forming an angle less or equal to $\pi/2$ with $\gamma'(0)$ (see \cite{4}). In \cite{12} Wang and Xia proved the following theorem

\begin{theorem}
   Given $\beta \in [0, 2]$, positive numbers $r_0$ and $C$, and an integer $ n \ge 2$, there is an
 $\epsilon = \epsilon(n, r_0 , C, \beta) > 0$ such that any complete Riemannian $n$-manifold $M$ with Ricci curvature $Ric_M \ge 0$,
$\alpha_M > 0$, $crit_p \ge r_0$ and 
$$k_p(r)\geq -\frac{C}{(1+r)^{\beta}},\ \frac{volB(p, r)}{\omega_{n}r^{n}}\le \left(1 +\frac{\epsilon}{r^{(n-2+\frac{1}{n})(1-\frac{\beta}{2})}}\right)\alpha_{M},$$
for some $p \in M$ and all $r \ge r_0$ is diffeomorphic to $\mathbb{R}^{n}.$
\end{theorem} 

In order to remove the condition of criticality radius, let us to define the function 
\begin{displaymath}\phi_\alpha(r)=\left\{\begin{array}{ll}
 r^\alpha, & \textrm{ for } \ r \ge 1,  \\
   r,& \textrm{ for } \ r < 1.
\end{array} 
\right. 
\end{displaymath}
We will prove a more general result

\begin{theorem}
Let $M$ be an $n$-dimensional $(n\ge 3)$ complete noncompact Riemannian manifold with
\begin{equation}
Ric_{(k)}(x)\ge -k\lambda(d_{p}(x)), \textrm{ for } 2\le k\le n-1,
\end{equation}
    and  
\begin{equation}
 K(x)\ge -\frac{C}{d_{p}(x)^\alpha},
\end{equation}
where $C(\lambda)=\int_{0}^{+\infty}s\lambda(s)ds < +\infty$ and $C>0$, $0\le \alpha\le 2$.
If $\alpha_p >0$, then there exists a constant 
$\epsilon=\epsilon(n, k,\lambda, C, \alpha)>0$, such that  $M$ is diffeomorphic to $\mathbb{R}^{n}$, provided that
\begin{equation}
 \frac{volB(p, r)}{vol B(\bar{p},r)}\le \left(1 +\frac{\epsilon\left(\phi_{\frac{1}{k+1}(\frac{k\alpha}{2}+1)}(r)\right)^{n-1}}{r^{n-1}}\right)\alpha_{p}, \ \textrm{ for  all } r>0.
\end{equation}
 
\end{theorem} 

In section 2, we will give some Abresch-Gromoll   excess function estimates for manifolds with $k$th  asymptotically nonnegative Ricci curvature. In section 3,  we will show that manifolds with suitable
ray desinty growth condition and curvature decay conditions are  diffeomorphic to $\mathbb{R}^{n}$ or
have finite topological type, and then using it to prove Theorem 1.2 and Theorem 1.5.

\section{Prelimanaries}
Let $M$ be an $n$-dimensional Riemannian manifold. For $p,q\in M$, the excess function $e_{pq}$ is defined by
$$e_{pq}(x)\dot= d_{p}(x) +d_{q}(x) - d(p,q).$$
Let $\gamma$ be a minimal geodesic from $p$ to $q$. If $Ric_{(k)}(x)\ge 0$ on all $x$ of $M$, Abresch-Gromoll \cite{2} (for $k=n-1$) and Shen \cite{8}(for any $k$) proved that
\begin{equation}
e_{pq}(x)\le 8 \left( \frac{s^{k+1}}{r}\right)^{1/k},
\end{equation}
where $s = d(x,\gamma), r=\min\{d(p,x),d(q,x)\}$.

For a manifold with $k$th  asymptotically nonnegative Ricci curvature, we will also give an estimate
for $e_{pq}(x)$. First we need

\begin{lemma}
Let $M$ be complete and $q, x\in M$. Suppose that $x$ is not on the cut locus of $q$, and
$$Ric_{(k)}(x)\ge -k\lambda(d_{p}(x)), \textrm{ with } C_0 = \int_{0}^{+\infty}s\lambda(s)ds < +\infty$$
along the minimal geodesic $\gamma$ from $x$ to $q$. Then for any orthonormal set
$\{ e_1, \cdots, e_{k+1}\}$ in $T_x M$ with $ \dot\gamma(0) \in span\{e_i\}$,
\begin{displaymath} \sum^{k+1}_{i=1}\nabla^2d_q(e_i,e_i)\le \left\{\begin{array}{ll}
\frac{1+\sqrt{1+8C_0}}{2}\cdot \frac{k}{d(p,x)}, & \textrm{ for } q=p,  \\
   \frac{k\sqrt{2C_0}}{ d(p,q)-d(q,x)}+\frac{k}{d(q,x)},& \textrm{ for } q\neq p, d(q,x)<d(p,q).
\end{array} 
\right. 
\end{displaymath}
\end{lemma}

\begin{proof}
For $q=p$, let $\gamma(s): [0, r] \to \mathbb{R}$ be the minimal normal geodesic  from $x$ to $p$. Since
$\dot\gamma(0) \in span\{e_i\}$, without loss of the generality, we may assume that 
$\textrm{grad} d_p(x)$=$\dot\gamma(0)=e_1$ and along $\gamma(t)$ have a orthonormal frame such that $e_i(r)=e_i$,
for $i=1, \cdots, k+1$. Put $N= \textrm{grad} d_p$, from \cite{3}, we have
\begin{eqnarray*}
&&\sum^{k+1}_{i=1}<R(e_i,N)N, e_i>\\& =& \sum^{k+1}_{i=2}<(\nabla_{e_i}\nabla_N-\nabla_N\nabla_{e_i}-\nabla_{[e_i, N]})N,e_i>\\
&= &-\sum^{k+1}_{i=2}N<\nabla_{e_i}, e_i>-\sum^{k+1}_{i=2}\sum^{n}_{j=2}<\nabla_{e_i}N,e_j><\nabla_{e_j}N, e_i>\\
&=& -\left(\sum^{k+1}_{i=2}h_{ii} \right)^\prime-\sum^{k+1}_{i=2}\sum^{n}_{j=2}h^2_{ij},\end{eqnarray*}
where $h_{ij}=<\nabla_{e_i}N,e_j>$ is the second fundamental form of the distance sphere from $p$.
From the Schwarz inequality and $k$th  asymptotically nonnegative Ricci curvature condition, we have
$$ -k(s)\lambda(s)\le -\left(\sum^{k+1}_{i=2}h_{ii} \right)^\prime-\frac{1}{k}\left(\sum^{k+1}_{i=2}h_{ii} \right)^2.$$
Note that,  $$ \sum^{k+1}_{i=2}h_{ii}(s) \sim \frac{k}{s}, \ \text{ as }s \to 0.$$
Consider the Riccati equation 
$$ v^\prime(s)+v^2(s)-\lambda(s)=0$$
satisfying
$$ v(s) \to \frac{1}{s}, \ \text{ as }s \to 0.$$
Standard comparison argument yields 
 $$\frac{1}{k} \sum^{k+1}_{i=2}h_{ii}(s) \le v(s).$$
From Lemma 3.4 in \cite{2}, we get
$$ v(r) \le \frac{1+\sqrt{1+8C_0}}{2r},$$
so we have 
\begin{equation} \sum^{k+1}_{i=1}\nabla^2d_q(e_i,e_i)= \sum^{k+1}_{i=2}h_{ii}(r)\le \frac{1+\sqrt{1+8C_0}}{2}\cdot \frac{k}{d(p,x)}.\end{equation}

For $q\neq p$ and $d(q,x)<d(p,q)$, using the similar argument and Lemma 3.2, 3.3 in \cite{2}, we have
\begin{equation} \sum^{k+1}_{i=1}\nabla^2d_q(e_i,e_i)\le\frac{k\sqrt{2C_0}}{ d(p,q)-d(q,x)}+\frac{k}{d(q,x)}.
\end{equation} 
\end{proof}

\begin{lemma}
Let $M$ be an $n$-dimensional ($n\ge 3$) complete Riemannian manifold   and let $\gamma$  be a minimal geodesic joining the base point $p$ and another point $q\in M$, $x\in M$ is a third point such that
$s<\min\{d(p,x),d(q,x), d(p,q)-d(q,x)\}$, where $s = d(x,\gamma)$.  Suppose $C_0 = \int_{0}^{+\infty}s\lambda(s)ds < +\infty$ and
$$Ric_{(k)}(x)\ge -k\lambda(d_{p}(x)), \textrm{ for } 2\le k\le n-1,$$
then
\begin{equation}
e_{pq}(x)\le  \frac{2k}{k-1}\frac{d(p,x)-s}{\sqrt{2C_0}s}\sinh \frac{\sqrt{2C_0}s}{d(p,x)-s}\left(\frac{C_2(s)}{2(k+1)}s^{k+1} \right)^{1/k},
\end{equation}
where $C_2(s)=\frac{1+\sqrt{1+8C_0}}{2} \frac{k}{d(p,x)-s}+\frac{k\sqrt{2C_0}}{ d(p,q)-d(q,x)-s}+\frac{k}{d(q,x)-s}$.
\end{lemma}

\begin{proof}
The argument is  using a modification methods of \cite{2} and \cite{10}. Denote
$S_\kappa(t)=\frac{\sinh \sqrt{-\kappa}t}{\sqrt{-\kappa}}$ for $\kappa <0$.
Take any  $s<R<\min\{d(p,x),d(q,x), d(p,q)-d(q,x)\}$ and  $C>C_2(R)$. Define $f:\overline{B(x,R)}\to \mathbb{R}$
as
$$f(y)=C\Phi_R(d_x(y))-e_{pq}(y), y\in \overline{B(x,R)},$$
where 
$$\Phi_R(\rho)=\int\int_{\rho\le t\le\tau\le R}\left( \frac{\dot S_\kappa(\tau)}{S_\kappa(t)}\right)^kd\tau dt.$$
Notice that from Lemma 3.3 in \cite{2}, the lower bound $\kappa$ on the Ricci curvature in the ball
$B(x,R)$ can be controlled by
$$Ric(y)\ge -(n-1)\frac{2C_0}{(d(p,x)-R)^2},  \  \forall y \in B(x,R), $$
Following the proof of Proposition 2.3 in \cite{2} and using Lemma 2.1, we can show that
for any $y\in\overline{B(x,R)}\setminus \{ x\}$, there is an orthonormal set
$\{ e_1, \cdots, e_{k+1}\}$ in $T_y M$ such that the following inequality holds in a
generalized sence:
\begin{eqnarray*}
&&\sum^{k+1}_{i=1}\nabla^2 f(e_i,e_i)\\& =& C\left(\Phi_R^{\prime\prime}\sum^{k+1}_{i=1}|\nabla_{e_i}d_x|^2+\Phi_R^{\prime} \sum^{k+1}_{i=1}\nabla^2d_x(e_i,e_i)\right)- \sum^{k+1}_{i=1}\nabla^2d_p(e_i,e_i)- \sum^{k+1}_{i=1}\nabla^2d_q(e_i,e_i)\\
&\ge & C-C_2(R)>0.
\end{eqnarray*}
Thus $f$ has no locally maximal point in $B(x, R)$. Since $f|_{S(x,R)}\le 0$ 
and $f|_{\overline{B(x,s)}\cap \gamma}> 0$, we know that
\begin{eqnarray*}
e_{pq}(x)&\le& \min_{0\le \rho\le R}\{\min_{y\in S(x,\rho)} e_{pq}(y)+2\rho\}\\
&\le& \min_{0\le \rho\le R}\{C \Phi_R(\rho)+2\rho\}\\
&\le& \min_{0\le \rho\le R}\left\{2\rho+\frac{C(S_{\kappa R^2}(1))^k}{2(k+1)}\left[\frac{2R^{k+1}}{k-1}(\rho^{1-k}-R^{1-k})+\rho^2-R^2 \right]\right\}\\
&\le &\frac{2k}{k-1}\frac{d(p,x)-R}{\sqrt{2C_0}R}\sinh \frac{\sqrt{2C_0}R}{d(p,x)-R}\left(\frac{CR^{k+1}}{2(k+1)} \right)^{1/k}.
\end{eqnarray*}
Let $R\to s$ and $C\to C_2(s)$, we get (2.4).
\end{proof}

Using Lemma 2.2  and an easy argument, we  can get a excess estimate for manifold with 
$k$th  asymptotically nonnegative Ricci curvature, which can be considered as an extended estimate of
Abresch-Gromoll  and Shen.
\begin{lemma}
Suppose $$Ric_{(k)}(x)\ge -k\lambda(d_{p}(x)), \textrm{ for } 2\le k\le n-1,$$
then

\begin{equation}
e_{pq}(x)\le 8 (1+8C_0)^{\frac{1+k}{2k}}\left( \frac{s^{k+1}}{r}\right)^{\frac{1}{k}},
\end{equation}
where $s = d(x,\gamma), r=\min\{d(p,x),d(q,x)\}$.

\end{lemma}

\section{proof of the theorems}

Let $R_{p}$ denotes the set of all ray issuing from $p$ and  $H(p,r) = \max_{x\in S(p,r)}d(x,R_{p}).$
  For a manifold with quadratic sectional curvature decay,  Wang and  Xia \cite{12} proved that there exists a constant $\epsilon$ such that if $H(p,r) < \epsilon r$, then it is  diffeomorphic to $\mathbb{R}^{n}$.
We will extend it to the following
\begin{theorem}
Given $C>0$, and $\alpha\in[0,2]$, suppose that $M$ is an n-dimensional
complete noncompact Riemannian manifold with
 $K(x)\geq -\frac{C}{d_{p}(x)^{\alpha}}$, then  there exists
 a positive constant $\epsilon=\epsilon(\alpha,C) $ such that if $H(p,r)< \epsilon \phi_{\frac{\alpha}{2}}(r)$, then $M$
 is diffeomorphic to $\mathbb{R}^{n}$.
 \end{theorem}
\begin{proof}
Let $\delta$ be a solution of the inequality $ \cosh ({2^{2+\alpha}\sqrt{C}}\epsilon )-\cosh^{2}({\frac{3}{2}2^{2+\alpha}\sqrt{C}}\epsilon)<0$
and take $\epsilon= \min\{\frac{1}{4}\delta, \frac{3}{32}\}$. From the Disk Theorem (cf. \cite{4}), it suffices to show that $d_p$ has no critical point other than $p$. Take an arbitrary point $x$($\neq p) \in M$ and let $r = d(p,x)$. 
Since $R_p$ is closed, there exists a ray $\gamma$ issuing from $p$ such that $s = d(x,\gamma)$. From
our condition, we have
\begin{equation}
s \le \epsilon \phi_{\frac{\alpha}{2}}(r).
\end{equation}
Let $q = \gamma (2r)$ and 
 $\sigma_{1}$ and $\sigma_{2}$ be geodesics joining $x$ to $p$ and $q$ respectively.
Set $\tilde{p} = \sigma_{1} (4\epsilon\phi_{\frac{\alpha}{2}}(r))$; $\tilde{q} = \sigma_{2} (4\epsilon\phi_{\frac{\alpha}{2}}(r))$. 
Consider
the triangle $\Delta( x,\tilde{p},\tilde{q})$, if $y$ is a point on this triangle, then
$$d(p,y)\ge d(p,x)-d(x,y)\ge d(p,x)-d(\tilde{p},x)-d(\tilde{p},y)$$
and
$$d(p,y)\ge d(p,x)-d(x,y)\ge d(p,x)-d(\tilde{q},x)-d(\tilde{q},y),$$
which means $d(p,y)\ge r-8 \epsilon \phi_{\frac{\alpha}{2}}(r)>\frac{r}{4}$.
Hence  the triangle $\Delta( x,\tilde{p},\tilde{q})\subset M\setminus B(p,\frac{r}{4}).$
Applying the Toponogov Theorem to the triangle $\Delta (x,\tilde{p},\tilde{q})$ we have:
\begin{equation}\label{p}
 \cosh \left(\frac{2^{\alpha}\sqrt{C}}{r^{\frac{\alpha}{2}}}d(\tilde{p},\tilde{q})\right)\le \cosh^{2}\left(\frac{2^{\alpha}\sqrt{C}}{r^{\frac{\alpha}{2}}}d(\tilde{p},x)\right)-\sinh^{2}\left(\frac{2^{\alpha}\sqrt{C}}{r^{\frac{\alpha}{2}}}d(\tilde{p},x)\right)\cos\theta,
\end{equation}
where $\theta =\angle\sigma_{1}'(0),\sigma_{2}'(0).$ 
Since $e_{pq}(x)\le 2s \le 2\epsilon \phi_{\frac{\alpha}{2}}(r)$,  from 
  triangle inequality, we have
\begin{equation}\begin{array}{rl}d(\tilde{p},\tilde{q})&\ge  d(p,q)-d(p,x)+d(\tilde{p},x)-d(x,q) +d(\tilde{q},x)\\&\ge 8\epsilon \phi_{\frac{\alpha}{2}}(r) - e_{pq}(x)\\
&\ge 6\epsilon \phi_{\frac{\alpha}{2}}(r).\end{array}\end{equation}
From (3.2), (3.3)   and (3.1) we have
\begin{eqnarray*}
&&\sinh^{2}\left(\frac{2^{\alpha}\sqrt{C}}{r^{\frac{\alpha}{2}}}d(\tilde{p},x)\right)\cos\theta\\ 
&\le &
\cosh^{2}\left(\frac{2^{2+\alpha}\sqrt{C}}{r^{\frac{\alpha}{2}}}\epsilon \phi_{\frac{\alpha}{2}}(r)\right)
-\cosh\left(\frac{2^{2+\alpha}\sqrt{C}}{r^{\frac{\alpha}{2}}}\frac{3\epsilon}{2 }\phi_{\frac{\alpha}{2}}(r)\right)\\
&\le& \cosh^2 \left({2^{2+\alpha}\sqrt{C}}\epsilon \right)-\cosh\left({\frac{3}{2}2^{2+\alpha}\sqrt{C}}\epsilon\right)\\
&<& 0,
\end{eqnarray*}
so
$$\theta >\frac{\pi}{2},$$
which shows that $x$ is not a critical ponit of $d_p$ and Theorem 3.1 follows.
\end{proof}

Before proving Theorem 1.2, we need the following lemma, which can be considered as a
generization of Lemma 3.1 in \cite{13}.

\begin{lemma}
Let $M$ be an $n$-dimensional $(n\ge 3)$ complete noncompact Riemannian manifold with
$$
Ric_{(k)}(x)\ge -k\lambda(d_{p}(x)), \textrm{ for } 2\le k\le n-1,
$$
    and  $$ K(x)\ge -\frac{C}{d_{p}(x)^\alpha},$$
where $C(\lambda)=\int_{0}^{+\infty}s\lambda(s)ds < +\infty$ and $C>0$, $0\le \alpha\le 2$.
There exists a constant 
$\epsilon^\prime=\epsilon^\prime( k,\lambda, C, \alpha)>0$, such that  $M$ has finite topological type, provided that
\begin{equation}
\limsup_{r\to +\infty }\frac{H(p,r)}{r^{\frac{1}{k+1}(\frac{k\alpha}{2}+1)}}
\leq {\epsilon}.
\end{equation}
\end{lemma}

\begin{proof}
Let $\delta$ be the solution of the inequality 
\begin{equation}
 \cosh^2 \left(2^{\alpha}\sqrt{C}\delta \right)-\cosh\left(\frac{3}{2}2^{\alpha}\sqrt{C}\delta\right)<0
\end{equation}
and take $\epsilon^\prime$ to be 
\begin{equation}
\epsilon^\prime= \left(\frac{\delta}{16(1+8C_0)^{\frac{1+k}{2k}}}\right)^{\frac{k}{k+1}}.
\end{equation}
 From (3.4), we can find a
constant $r_0 >1$ such that
\begin{equation}
{H(p,r)}
\leq {\epsilon^\prime}{r^{\frac{1}{k+1}(\frac{k\alpha}{2}+1)}}, \ \forall r \ge r_0.
\end{equation}

It suffices to show that $d_p$ has no critical point in $M\setminus B(p, r_0)$. To do this, take an arbitrary point $x \in M\setminus B(p, r_0)$ and let $r = d(p,x)$. 
 From
our condition, there exists a ray $\gamma$ issuing from $p$ such that $s = d(x,\gamma)$ and
\begin{equation}
s \le \epsilon^\prime {r^{\frac{1}{k+1}(\frac{k\alpha}{2}+1)}} <r.
\end{equation}
Let $q = \gamma (2r)$ and 
 $\sigma_{1}$ and $\sigma_{2}$ be geodesics joining $x$ to $p$ and $q$ respectively.
Set ${p}^\prime = \sigma_{1} (\delta r^{\frac{\alpha}{2}})$; ${q}^\prime = \sigma_{2} (\delta r^{\frac{\alpha}{2}})$. As in the proof of Theorem 3.1, we know that
the triangle $\Delta( x,{p}^\prime,{q}^\prime)\subset M\setminus B(p,\frac{r}{4})$
and 
\begin{equation}
 d({p}^\prime, {q}^\prime ) \ge 2 \delta r^{\frac{\alpha}{2}}-e_{pq}(x).
\end{equation}
Using (2.5), (3.6)   and (3.8) we have
$$
 \begin{array}{rl}e_{pq}(x)&\le  8 (1+8C_0)^{\frac{1}{2k}}\left( \frac{s^{k+1}}{r}\right)^{\frac{1}{k}}\\
&\le  8 (1+8C_0)^{\frac{1}{2k}}\left( \frac{{(\epsilon^\prime})^{k+1}r^{k\alpha/2+1}}{r}\right)^{\frac{1}{k}}\\
&=\frac{\delta}{2}r^{\frac{\alpha}{2}}.\end{array}
$$
So we have
\begin{equation}
 d({p}^\prime, {q}^\prime ) \ge 2 \delta r^{\frac{\alpha}{2}}-\frac{1}{2}\delta r^{\frac{\alpha}{2}}
=\frac{3}{2}\delta r^{\frac{\alpha}{2}}.
\end{equation}
Applying the Toponogov Theorem to the triangle $\Delta (x,p^\prime,q^\prime)$ we have
\begin{eqnarray*}
&&\sinh^{2}\left(\frac{2^{\alpha}\sqrt{C}}{r^{\frac{\alpha}{2}}}d ({p}^\prime,x)\right)\cos\theta\\ 
&\le &
\cosh^{2}\left(\frac{2^{\alpha}\sqrt{C}}{r^{\frac{\alpha}{2}}}d ( {p}^\prime,x)\right)
-\cosh\left(\frac{2^{\alpha}\sqrt{C}}{r^{\frac{\alpha}{2}}}
d(p^\prime,q^\prime)\right)\\
&\le& \cosh^2 \left(2^{\alpha}\sqrt{C}\delta \right)-\cosh\left(\frac{3}{2}2^{\alpha}\sqrt{C}\delta\right)\\
&<& 0,
\end{eqnarray*}
so
$$\theta >\frac{\pi}{2},$$
which shows that $x$ is not a critical ponit of $d_p$ and Lemma 3.2 follows.
\end{proof}

\noindent{\it Proof of Theorem 1.2.} Take the number $\epsilon$ in the Theorem 1.2 to be
\begin{equation}\epsilon=\frac{(\epsilon^\prime)^{n-1}}{ 18^{n}e^{3(n-1)C_{0}}},
\end{equation}
where $\epsilon^\prime=\epsilon^\prime(k,\lambda, C, \alpha)$ is as in Lemma 3.2.
 From (1.4), we can find a
constant $r_0 >1$ such that
\begin{equation}
\frac{vol B(x, r)}{vol B(\bar{x},r)}-\alpha_p
\leq {\epsilon}\alpha_pr^{-\frac{(n-1)k}{k+1}(1-\frac{\alpha}{2})}, \ \forall r \ge r_0.
\end{equation}
From Lemma 3.2,  it only need to show that for any arbitrary point $x \in M\setminus B(p, r_0)$ and  a ray $\gamma$ issuing from $p$,  set $r=d(p, x)$ and $s = d(x,\gamma)$,  then
\begin{equation}
s \le \epsilon^\prime {r^{\frac{1}{k+1}(\frac{k\alpha}{2}+1)}}.
\end{equation}
To prove it, let $\Sigma_{p}(\infty)$ be the set of unit vectors
$v\in S_pM$ such that the geodesic $\gamma(t)=exp_p(tv)$ is a ray and  $\Sigma_{p}^{c}(\infty)=S_{p}\setminus\Sigma_{p}(\infty)$. We have 
$$ B( x,s)\subset B_{\Sigma_{p}^{c}(\infty)}(p,r+s)\setminus B(p,r-s),$$
which means \begin{equation}
 volB( x,s)\le volB_{\Sigma_{p}^{c}(\infty)}(p,r+s)-volB(p,r-s).
 \end{equation}
By the Relative Comparison Theorem for asymptotically nonnegative Ricci curvature (see \cite{6}), we have
\begin{eqnarray}
 volB( x,\frac{s}{2})&\le  & volB_{\Sigma_{p}^{c}(\infty)}(p,r+\frac{s}{2})-volB_{\Sigma_{p}^{c}(\infty)}(p,r-\frac{s}{2})\nonumber\\
&= &volB_{\Sigma_{p}^{c}(\infty)}(p,r-\frac{s}{2})\left( \frac{volB_{\Sigma_{p}^{c}(\infty)}(p,r+\frac{s}{2})}{volB_{\Sigma_{p}^{c}(\infty)}(p,r-\frac{s}{2})} -1\right)\nonumber\\
&\le &volB_{\Sigma_{p}^{c}(\infty)}(p,r-\frac{s}{2})\left( \frac{volB(\bar{p},r+\frac{s}{2})}{volB(\bar{p},r-\frac{s}{2})} -1\right)\nonumber\\
&\le &e^{(n-1)C_{0}}\left(\left(\frac{r+\frac{s}{2}}{r-\frac{s}{2}}\right)^{n}-1\right)volB_{\Sigma_{p}^{c}(\infty)}(p,r-\frac{s}{2})\\
&\le &\left(3^{n}-1\right)e^{(n-1)C_{0}}\frac{s}{r}volB_{\Sigma_{p}^{c}(\infty)}(p,r-\frac{s}{2}),\end{eqnarray}
where the last two inequalities is in fact due to Mahaman by using the volume element estimate  $dv(t) \le e^{(n-1)C_{0}}t^{n-1}dt\wedge dS_{n-1}$ in polar
coordinates.

Now, from (3.14), Lemma  3.10 in \cite{6} and (3.12), we have

\begin{eqnarray}  volB_{\Sigma_{p}^{c}(\infty)}(p,r-\frac{s}{2})&=& volB(p,r-\frac{s}{2})-volB_{\Sigma_{p}(\infty)}(p,r-\frac{s}{2})\nonumber\\
&\le & volB(p,r-\frac{s}{2})-\alpha_{p}volB(\overline{p},r-\frac{s}{2})\nonumber\\
&\le& {\epsilon}\alpha_p r^{-\frac{(n-1)k}{k+1}(1-\frac{\alpha}{2})}volB(\bar{p},r-\frac{s}{2})\nonumber\\
&\le&{\epsilon}\omega_n\alpha_p e^{(n-1)C_{0}}r^{-\frac{(n-1)k}{k+1}(1-\frac{\alpha}{2})}r^n.
\end{eqnarray}
Subsituting (3.17) in (3.14), we get
\begin{equation}
 volB( x,\frac{s}{2})\le (3^n-1){\epsilon}\omega_n \alpha_pe^{2(n-1)C_0}sr^{n-1-\frac{(n-1)k}{k+1}(1-\frac{\alpha}{2})}.
\end{equation}

On the other hand, using (15) in \cite{6}, we know that
\begin{equation}
 volB( x,\frac{s}{2})\ge \frac{\omega_n}{3^n} \alpha_pe^{-(n-1)C_0} \left(\frac{s}{2}\right)^n.
\end{equation}
From these  two inequalities, we have
\begin{equation}
s^{n-1}\leq\epsilon 18^{n}e^{3(n-1)C_{0}} r^{\frac{n-1}{k+1}(\frac{k\alpha}{2}+1)},\end{equation}
using (3.12), we obtain
$$ s\le (\epsilon 18^{n}e^{3(n-1)C_{0}})^{\frac{1}{n-1}}r^{\frac{1}{k+1}(\frac{k\alpha}{2}+1)}=
\epsilon^\prime r^{\frac{1}{k+1}(\frac{k\alpha}{2}+1)},
$$
which satisfying (3.13) and completing the proof of Theorem 1.2. $\hspace{22mm} \square$

\vspace{2mm}

\noindent{\it Proof of Theorem 1.5.}  We choose the number $\epsilon$ in the Theorem 1.5 to be
\begin{equation}\epsilon=\frac{(\epsilon^\prime)^{n-1}}{ 18^{n}e^{3(n-1)C_{0}}},
\end{equation}
where $\epsilon^\prime=\epsilon^\prime(k,\lambda, C, \alpha)$ is as in Lemma 3.2. Take any arbitrary point $x \in M\setminus\{p\}$ and  a ray $\gamma$ issuing from $p$, and set $r=d(p, x)$ and $s = d(x,\gamma)$. Using the similar methods as in the proof of Theorem 1.2 and our condition,
\begin{equation}
 volB( x,\frac{s}{2})\le (3^n-1){\epsilon}\omega_n \alpha_pe^{2(n-1)C_0}s\left(\phi_{\frac{1}{k+1}(\frac{k\alpha}{2}+1)}(r)\right)^{n-1}.
\end{equation}
From (3.19) and (3.22), we obtain
\begin{equation}
s \le \epsilon^\prime \phi_{\frac{1}{k+1}(\frac{k\alpha}{2}+1)}(r).
\end{equation}
 Repeatting the argument as in the proof of Theorem 3.1 and Lemma 3.2, we
can show that $x$ is not a critical ponit of $d_p$ and Theorem 1.5 follows.

\noindent { \small {\bf Acknowledgements}  The author would like to thank referees for their valuable suggestions and remarks.}

\bibliographystyle{amsplain}

\end{document}